\documentclass[11pt, a4paper, english]{amsart}
\usepackage{amsmath} 
\usepackage{amsthm} 
\usepackage{amssymb} 
\usepackage[dvipsnames]{xcolor}
\usepackage{amscd} 
\usepackage{mathtools}
\usepackage{booktabs}
\usepackage[boxed]{algorithm2e}
\usepackage{subcaption}

\usepackage{array} 
\usepackage[cal=boondoxo,scr=euler]{mathalfa}
\usepackage[backref=page,linktocpage]{hyperref} 
\usepackage{cleveref} 
\usepackage{caption} %
\usepackage{graphics,graphicx} 
\usepackage{tikz,tikz-cd} 

\usetikzlibrary{graphs,graphs.standard,calc}

\usetikzlibrary{decorations.pathreplacing,}

\usepackage{enumerate} 

\DeclareMathAlphabet{\mathsf}{OT1}{\sfdefault}{m}{n}

\newcommand{\nocontentsline}[3]{}
\newcommand{\tocless}[2]{\bgroup\let\addcontentsline=\nocontentsline#1{#2}\egroup}

\usepackage[margin=1.20in]{geometry} 

\usepackage{verbatim}

\usepackage{scalerel}

\makeatletter
\def\dual#1{\expandafter\dual@aux#1\@nil}
\def\dual@aux#1/#2\@nil{\begin{tabular}{@{}c@{}}#1\\#2\end{tabular}}
\makeatother

\makeatletter
\@namedef{subjclassname@2020}{\textup{2020} Mathematics Subject Classification}
\makeatother

\DeclareMathAlphabet{\amathbb}{U}{bbold}{m}{n}

\hypersetup{
    colorlinks = true,
    linkbordercolor = {white},
    linkcolor = {NavyBlue},
    anchorcolor = {black},
    citecolor = {NavyBlue},
    filecolor = {cyan},
    menucolor = {NavyBlue},
    runcolor = {cyan},
    urlcolor = {NavyBlue}
}

\usetikzlibrary{automata}

\newtheoremstyle{teoremas}
{12pt}
{13pt}
{\itshape}
{}
{\bfseries}
{}
{.5em}
{}

\theoremstyle{teoremas}
\newtheorem{theorem}{Theorem}[section]
\newtheorem{corollary}[theorem]{Corollary}
\newtheorem{lemma}[theorem]{Lemma}
\newtheorem{proposition}[theorem]{Proposition}

\newtheoremstyle{definition}
{12pt}
{12pt}
{}
{}
{\bfseries}
{}
{.5em}
{}

\theoremstyle{definition}
\newtheorem{definition}[theorem]{Definition}
\newtheorem{conjecture}[theorem]{Conjecture}

\newtheorem{example}[theorem]{Example}
\newtheorem{remark}[theorem]{Remark}

\crefname{theorem}{theorem}{theorems}
\Crefname{theorem}{Theorem}{Theorems}
\crefname{lemma}{lemma}{lemmas}
\Crefname{lemma}{Lemma}{Lemmas}
\crefname{proposition}{proposition}{propositions}
\Crefname{proposition}{Proposition}{Propositions}

\DeclareMathOperator{\rk}{rk}

\newcommand{\M}{\mathsf{M}}
\newcommand{\N}{\mathsf{N}}
\newcommand{\U}{\mathsf{U}}

\newcommand{\K}{\mathsf{K}}

\renewcommand{\path}{\operatorname{path}}

\newcommand{\ent}{\alpha}

\AtBeginDocument{%
   \def\MR#1{}
}

\title[Schubert matroids, Delannoy paths, and Speyer's invariant]{Schubert matroids, Delannoy paths,\\ 
and Speyer's invariant}

\author[L.~Ferroni]{Luis Ferroni}
\thanks{The author is supported by the Swedish
Research Council grant 2018-03968. }
\address{(L. Ferroni)
  Department of Mathematics, KTH Royal Institute of Technology, Stockholm, Sweden
}
\email{ferroni@kth.se}


\allowdisplaybreaks

\begin{document}

\begin{abstract}
    We provide a combinatorial way of computing Speyer's $g$-polynomial on arbitrary Schubert matroids via the enumeration of certain Delannoy paths. We define a new statistic of a basis in a matroid, and express the $g$-polynomial of a Schubert matroid in terms of it and internal and external activities. Some surprising positivity properties of the $g$-polynomial of Schubert matroids are deduced from our expression. Finally, we combine our formulas with a fundamental result of Derksen and Fink to provide an algorithm for computing the $g$-polynomial of an arbitrary matroid. 
\end{abstract}

\maketitle

{\hfill\footnotesize\emph{Para Bruno. Gracias por hacerme papá.}}

\section{Introduction}\label{sec:one}

\subsection{Overview}

One of the most pervasive objects within combinatorial theory is the hypersimplex $\Delta_{k,n}$. Many features of these polytopes have been studied throughout the 
literature; for example, their volumes \cite{stanleyeulerian}, unimodular triangulations \cite{lam-postnikov}, $f$-vectors \cite{hibi-li-ohsugi}, Ehrhart and $h^*$-polynomials \cite{nanli,kim-ehrhart,ferroni1}, and polytopal subdivisions related to them \cite{early,olarte-santos}. A well known occurrence of $\Delta_{k,n}$ in algebraic combinatorics is as the base polytope of the uniform matroid of rank $k$ on $n$ elements $\U_{k,n}$. The base polytope of every matroid of rank $k$ and cardinality $n$ can be seen as a subpolytope of the hypersimplex $\Delta_{k,n}$. Of particular interest in this paper will be the regular subdivisions of $\Delta_{k,n}$ into smaller matroid polytopes; these subdivisions are parameterized by a subfan of the secondary fan of $\Delta_{k,n}$, commonly known as the Dressian $\operatorname{Dr}(k,n)$. Dressians and regular matroid subdivisions are prominent objects within the tropical geometry framework, see \cite{herrmann-jensen-joswig-sturmfels,herrmann-joswig-speyer,joswig-schroter,OrlatePanizzuitSchroeter,speyer-williams}.

A prominent conjecture regarding matroid subdivisions was posed by Speyer in \cite{speyer-conjecture}.

\begin{conjecture}[\cite{speyer-conjecture}]\label{conj:f-vector-conj}
    Let $\mathcal{S}$ be a subdivision of $\Delta_{k,n}$ into smaller matroid polytopes. For each $1\leq i\leq n$, denote by $f_i$ the number of cells of $\mathcal{S}$ of dimension $n-i$ lying in the interior of $\Delta_{k,n}$. Then:
        \[ f_i \leq \binom{n-1-i}{k-i}\binom{n-k-1}{i-1}.\]
    Moreover, the simultaneous equality case occurs if and only if all the facets of $\mathcal{S}$ correspond to base polytopes of series-parallel matroids.
\end{conjecture}

A \emph{series-parallel matroid} is a matroid that can be obtained from $\U_{0,1}$ or $\U_{1,1}$ via a sequence of series or parallel extensions. This family consists of the single loop $\U_{0,1}$, the single coloop $\U_{1,1}$ and all the matroids whose $\beta$-invariant is equal to $1$ (in particular, we are using the convention that series-parallel matroids are connected).

The above conjecture, known as ``the $f$-vector conjecture'', is known be true in a number of cases. Most notably, the conjecture holds whenever all the internal cells of the subdivision $\mathcal{S}$ correspond to matroids realizable over a field of characteristic $0$; that is precisely the content of another result of Speyer in \cite{speyer}. The proof relies on a deep result from algebraic geometry known as the Kawamata--Viehweg vanishing theorem; unfortunately, it is not clear how to extend this vanishing result even to matroids representable over a field of positive characteristic.

One of the main players in Speyer's proof of the aforementioned instance of Conjecture~\ref{conj:f-vector-conj} is a matroid invariant known as the $g$-polynomial. This invariant was originally defined for matroids representable over $\mathbb{C}$ by Speyer in \cite{speyer} via the $\mathrm{K}$-theory of the Grassmannian, and later to all matroids by Fink and Speyer in \cite{fink-speyer} via equivariant localization.

As Speyer mentions in \cite[p.~887]{speyer}, the coefficients of the $g$-polynomial of a matroid $\M$ are ``morally'' counting the number of base polytopes of (direct sums of) series-parallel matroids that are needed to built the base polytope $\mathscr{P}(\M)$. This can be made rigorous as follows.

\begin{definition}\label{defi:g-poly}
    There is a unique way of associating to each matroid $\M$ a polynomial invariant $g_{\M}(t)\in\mathbb{Z}[t]$ in such a way that the following properties hold:
    \begin{enumerate}[(i)]
        \item If $\M$ has loops or coloops, then $g_{\M}(t) = 0$.
        \item If $\M$ is a series-parallel matroid on at least two elements, then $g_{\M}(t) = t$.
        \item If $\M= \M_1\oplus \M_2$, then $g_{\M}(t) = g_{\M_1}(t)\cdot g_{\M_2}(t)$.
        \item The map $\M \longmapsto g_{\M}(t)$ is a covaluation under matroid polytope subdivisions.
    \end{enumerate}
    The polynomial $g_{\M}(t)$ is referred to as the \emph{$g$-polynomial} of the matroid $\M$.
\end{definition}

The fact that there exists at least one invariant satisfying the above conditions follows from the work of Speyer \cite{speyer} and Fink and Speyer \cite[Section~4]{fink-speyer}. The fact that there exists at most one invariant satisfying the above conditions follows from the fact that direct sums of series-parallel matroids span the covaluative group of matroid polytopes; this in turn can be seen as a consequence of the reasoning of Ferroni and Schr\"oter in the proof of \cite[Theorem~5.22]{ferroni-schroter2}. Recall that Derksen and Fink proved in \cite{derksen-fink} that the class of Schubert matroids is a basis of the covaluative group of matroid polytopes, and on the other hand the base polytope of every Schubert matroid admits a matroid subdivision in which all the interior cells are direct sums of series-parallel matroids.

There exists an alternative description of the $g$-polynomial via the tautological classes of matroids of Berget, Eur, Spink and Tseng \cite[Theorem~10.12]{best}, which also proves a Chow-theoretic formula for $g_{\M}(t)$ previously conjectured by L\'opez de Medrano, Rinc\'on and Shaw in \cite{lopezdemedrano-rincon-shaw}. We mention that there is a generalization of the $g$-polynomial for morphisms of matroids in the work of Dinu, Eur and Seynnaeve \cite{dinu-eur-seynnaeve}.

A strengthening of Conjecture~\ref{conj:f-vector-conj}, also due to Speyer is the following:

\begin{conjecture}[\cite{speyer}]
    For every matroid $\M$, the polynomial $g_{\M}(t)$ has non-negative coefficients.
\end{conjecture}

This conjecture is known to hold for all matroids representable over a field of characteristic $0$ \cite[Proposition~3.3]{speyer} and for all sparse paving matroids \cite[Theorem~13.16]{ferroni-schroter2}. Let us denote by $c(\M)$ the number of connected components of $\M$. It is known that if $\M$ does not possess loops nor coloops, then $[t^{c(\M)}]g_{\M}(t)$ is strictly positive, as it is the product of the $\beta$-invariants of the connected components of $\M$. To explain why the last conjecture implies the $f$-vector conjecture, let us consider any subdivision $\mathcal{S}$ of $\Delta_{k,n}$. The fact that the $g$-polynomial is covaluative yields:
    \[ g_{\U_{k,n}}(t) = \sum_{\mathscr{P}(\N)\in\mathcal{S}^{\circ}} g_{\N}(t),\]
where $\mathcal{S}^{\circ}$ consists of all the internal faces of the subdivision $\mathcal{S}$. In particular, assuming that the $g$-polynomials of all the matroids $\N$ appearing in the subdivision are non-negative, given that the coefficient of $t^{c(\N)}$ is strictly positive, one has the coefficient-wise inequality:
     \[ g_{\U_{k,n}}(t) \succeq \sum_{\mathscr{P}(\N)\in\mathcal{S}^{\circ}} t^{c(\N)} = \sum_{i=1}^{n} \#\{\N : \mathscr{P}(\N)\in\mathcal{S}^{\circ}\text{ and }\dim \mathscr{P}(\N) = n-i\}\, t^i, \]
whereas the left-hand-side can be explicitly computed (see, e.g., equation~\eqref{eq:g-uniform} below), yielding precisely the inequality predicted by Conjecture~\ref{conj:f-vector-conj}.
\subsection{Main results}

One of the major obstacles of working with the $g$-polynomial is that it is is undoubtedly very hard to compute for general matroids. The original definitions of Speyer \cite{speyer} and Fink and Speyer \cite{fink-speyer}, as well as the formula of Berget, Eur, Spink and Tseng \cite{best} lend themselves very well for theoretical purposes; however, they are fairly complicated to turn into a pseudocode allowing a computer to perform the calculations. Arguably, the ultimate motivation in this paper is to resolve this issue. 

As mentioned before, Schubert matroids are a basis for the covaluative group of matroid polytopes. They constitute the fundamental blocks throughout our procedure. Our first main contribution is giving a combinatorial interpretation via the enumeration of ``admissible'' Delannoy paths of the coefficients of the $g$-polynomial of arbitrary Schubert matroids.

\begin{theorem}\label{thm:main-delannoy-g}
    Let $\M$ be a loopless and coloopless Schubert matroid. The $g$-polynomial of $\M$ is given by:
        \[ g_{\M}(t) = \sum_{i=1}^{\rk(\M)} c_i\, t^i,\]
    where $c_i$ counts the number of admissible Delannoy paths associated to $\M$ having exactly $i-1$ diagonal steps. 
\end{theorem}

This retrieves the non-negativity of the coefficients of the $g$-polynomial of Schubert matroids (which also follows from Speyer's \cite[Proposition~3.3]{speyer}, as they are representable over $\mathbb{C}$). This has interesting consequences. As a glimpse, one is able to derive a very short proof of the formula of the $g$-polynomial of uniform matroids, or prove that the coefficients of the $g$-polynomial of Catalan matroids \cite{ardila-catalan} match with the $f$-vectors of associahedra.

By abstracting our definition of admissible Delannoy paths and removing all the lattice path terminology, we can provide an equivalent statement in terms of statistics of bases in ordered matroids. Two of the three players are well-known, they are the internal and external activity of a basis $B$, usually denoted $i(B)$ and $e(B)$, respectively. The third is new, and we denote it by $\alpha(B)$.

\begin{definition}
    Let $\M=(E,\mathscr{B})$ be a matroid on an (ordered) ground set $E$. For simplicity, assume that $E=[n]$ and that the order is given by $1<2<\cdots<n$. For each basis $B$ of $\M$ we define
    \[ \alpha(B):= \#\left\{i\in B : B':=(B\smallsetminus\{i\})\cup\{i+1\}\in \mathscr{B} \text{ and } e(B') = e(B), \,i(B')=i(B)\right\}.  \]
\end{definition}

An equivalent reformulation of Theorem~\ref{thm:main-delannoy-g} in terms of these notions is as follows.

\begin{theorem}\label{thm:main-entropies}
    Let $\M$ be a loopless and coloopless Schubert matroid. Then, the $g$-polynomial of $\M$ is given by
    \[ g_{\M}(t) = t\; \sum_{\substack{B\in\mathscr{B}\\e(B)=1\\i(B)=0}} (t+1)^{\ent(B)}. \]
\end{theorem}

This statement is particularly useful to calculate with a computer the $g$-polynomial of Schubert matroids. We point out, however, that the above formula does not work for matroids in general. As mentioned before, the positivity of the coefficients of the $g$-polynomials of Schubert matroids is not new. What is more striking is that a stronger positivity phenomenon holds within this class.

\begin{corollary}\label{coro:positivity-g-tilde}
    For every matroid, let us denote $\widetilde{g}_{\M}(t) := \frac{1}{t}g_{\M}(t)$. If $\M$ is a Schubert matroid we have that $\widetilde{g}_{\M}(t-1)$ has non-negative coefficients.
\end{corollary}

The assumption on $\M$ being Schubert is essential, as there are non-Schubert matroids for which the above statement fails. Counterexamples come in two distinct flavors. On one hand, it is easy to construct disconnected matroids for which the above property fails: for instance $\M=\U_{1,2}\oplus\U_{1,2}$ has two direct summands that are series-parallel, thus $g_{\M}(t) = t^2$. In particular $\widetilde{g}_{\M}(t) = t$ and therefore $\widetilde{g}_{\M}(t-1)=t-1$ which fails to have non-negative coefficients. Within the realm of connected matroids, it is more challenging to find examples for which $\widetilde{g}_{\M}(t-1)$ attains negative coefficients. In fact, the smallest such example is precisely the Fano matroid, which is coincidentally the smallest matroid that is not representable over a field of characteristic $0$. Using the methods of Ferroni and Schr\"oter \cite{ferroni-schroter2} it is not difficult to prove that sparse paving matroids\footnote{Sparse paving matroids on at least $5$ elements and rank/corank greater than $1$ are always connected.} with sufficiently many non-bases may easily attain a negative coefficient. In particular, the sparse paving matroid known as $\mathsf{R}_8$ (see \cite[p.~646]{oxley}) is representable over the complex numbers and $\widetilde{g}_{\mathsf{R}_{8}}(t-1)$ attains a negative coefficient. We conjecture, however, that Corollary~\ref{coro:positivity-g-tilde} admits an extension to all matroids that can be subdivided into series-parallel matroids. 

\subsection{Outline} 

In Section~\ref{sec:two} we make a quick recapitulation of some useful notions regarding Schubert matroids, lattice path matroids,  and covaluations that we will be using throughout the paper. In Section~\ref{sec:three} we define the notion of ``admissible Delannoy path'' for Schubert matroids and prove Theorem~\ref{thm:main-delannoy-g} (it is stated as Theorem~\ref{thm:main-delannoy-g-body}), and we discuss some immediate consequences of this result for some particular Schubert matroids. In Section~\ref{sec:four} we provide a lattice-path-free reformulation of the notion of admissibility of Delannoy paths, motivate the statistic $\alpha(B)$, and prove Theorem~\ref{thm:main-entropies} (it corresponds to Theorem~\ref{thm:main-entropies-body}). In Section~\ref{sec:five} we use the methods developed throughout in combination with results of Derksen and Fink \cite{derksen-fink} and Hampe \cite{hampe}, to provide a way of computing the $g$-polynomial of an \emph{arbitrary} matroid. In Table~\ref{tab:hvec} we list the $g$-polynomials of many famous or relevant matroids, and provide a pseudocode together with its \texttt{SAGE} implementation for computing $g$-polynomials of matroids.

\section{Preliminaries}\label{sec:two}

Throughout this paper we will assume that the reader is familiar with most of the terminology and constructions in classical matroid theory, for which we refer to Oxley's book \cite{oxley} and White's anthologies \cite{white2,white1,white3}. Let us make a brief review of some additional notions that we will require.

\subsection{Schubert matroids}

Let us consider a finite set $E$ and a total ordering $<$ on $E$. Consider two $r$-subsets $I=\{i_1<i_2<\cdots<i_r\}$ and $J=\{j_1<j_2<\cdots<j_r\}$. We will write $I\leq J$ if $i_{\ell}\leq j_{\ell}$ for each $\ell=1,\ldots,r$.

\begin{definition}
    A \emph{Schubert matroid} on $E$ of rank $r$ is a matroid whose set of bases is given by:
        \begin{equation}
        \{B\subseteq E: |B| = r \text{ and } U\leq B\}\label{eq:schubert}
        \end{equation}
    for some total ordering $<$ on $E$ and some set $U\subseteq E$ of cardinality $r$.
\end{definition}

Whenever we say that an ordered matroid $\M$ is Schubert, then we will be tacitly implying that the ordering of the ground set of $\M$ is precisely the one mentioned in the preceding definition.

Let us mention briefly that the definition of Schubert matroids stated above is essentially borrowed from \cite[Definition~7.5]{eur-huh-larson}. It is equivalent to saying that Schubert matroids are precisely the matroids whose lattice of cyclic flats is a chain \cite[Definition~2.20]{ferroni-schroter2}. Other sources use the naming \emph{nested matroids}, \emph{shifted matroids} and \emph{generalized Catalan matroids}. 

\subsection{Matroids and lattice paths}

A convenient feature of Schubert matroids is that they can be represented as lattice path matroids by using skew shapes. For more background on lattice path matroids we refer to \cite{bonin-demier}. 

Fix two integers $0\leq r\leq n$. Let us consider lattice paths in $\mathbb{R}^2$ starting at $(0,0)$, ending at $(n-r,r)$ and consisting of steps of the form $+(1,0)$ (an ``east step'') or $+(0,1)$ (a ``north step''). Each such path consists of exactly $r$ north steps and exactly $n-r$ east steps, in particular it is straightforward to check that there are $\binom{n}{r}$ such paths. Clearly a path will be determined if one knows the positions in which a north step is performed. For every $P=\{p_1<\cdots<p_r\}$, we will denote by $\path(P)$ the unique path having north steps at the positions $p_1,\ldots,p_r$ (and east positions elsewhere).

Let us consider as ground set $E=[n]$ endowed with the usual ordering of the positive integers, and consider an $r$-set $U=\{u_1<\cdots<u_r\}$. We can consider a lattice path representation of $U$ using the numbers $u_1,\ldots, u_r$ as the positions of the north steps. The family of all sets in equation \eqref{eq:schubert} corresponds to the family of all lattice paths that lie below $U$. This is the \emph{lattice path matroid presentation} of $U$. For example, in Figure~\ref{fig:lattice-path} we have $n=12$, $r=5$ and $U=\{1,2,5,7,10\}$, the highlighted path is $\path(U)$.

\begin{center}
    \begin{figure}[ht]
        \begin{tikzpicture}[scale=0.65, line width=.5pt]
        \draw[line width=2pt,line cap=round] (0,0)--(0,2)--(2,2)--(2,2)--(2,3)--(3,3)--(3,4)--(5,4)--(5,5)--(6,5)--(7,5);
        \foreach \x/\y/\l in {-0.5/0.5/1, -0.5/1.5/2, 0.5/2.3/3, 1.5/2.3/4, 2.3/2.5/5, 2.5/3.3/6, 3.3/3.5/7, 3.5/4.3/8, 4.5/4.3/9, 5.3/4.5/10, 5.5/5.3/11, 6.5/5.3/12}
        	\node at (\x,\y) {\scriptsize\sf\l};
        \draw (0,0) grid (7,2);
        \draw (2,2) grid (7,3);
        \draw (3,3) grid (7,4);
        \draw (5,4) grid (7,5);
        \draw[decoration={brace,mirror, raise=8pt},decorate]
         (0,0) -- node[below=10pt] {$n-r$} (7,0);
        \draw[decoration={brace,mirror, raise=8pt},decorate]
         (7,0) -- node[right=10pt] {$r$} (7,5);
        \end{tikzpicture}
    \caption{$n=12$, $r=5$ and $U=\{1,2,5,7,10\}$.}\label{fig:lattice-path}
    \end{figure}
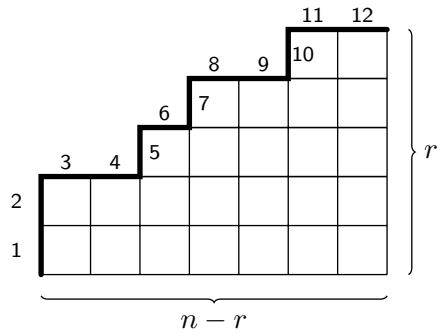
\end{center}

The class of lattice path matroids arises when varying the ``lower path''. Whenever we are dealing with a Schubert matroid on $[n]$ of rank $r$ using the ordering $1<\cdots < n$, we will denote $L=\{n-r+1,\ldots,n\}$ (this corresponds to the lexicographically maximal basis of the Schubert matroid).

\begin{proposition}[{\cite[Theorem~5.3]{bonin-demier}}]\label{prop:basis-activities-schubert}
    Let $E=[n]$, $U=\{u_1<\cdots<u_r\}$ and let $\M$ be the corresponding Schubert matroid. For each basis $B$ of $\M$, the internal and external activities are given by:
    \begin{align*}
        i(B) &= \#\{ \text{intersections of $\path(B)$ with $\path(U)$ at a north step}\}\\
        e(B) &= \#\{ \text{intersections of $\path(B)$ with $\path(L)$ at an east step}\}.
    \end{align*}
    Equivalently,
    \begin{align*}
        i(B) &= |B\cap U|\\
        e(B) &= n - |B\cap L|.
    \end{align*}
\end{proposition}

\begin{example}
    Consider again the Schubert matroid of Figure~\ref{fig:lattice-path}, and let $B = \{3,4,5,7,12\}$, depicted in Figure~\ref{fig:activities}. Notice that $\path(B)$ intersects $\path(L)$ in the first two steps and in the last one. Only the first two steps are east steps, hence $e(B)=2$. On the other hand, $\path(B)$ intersects $\path(U)$ at the steps number $5$, $6$, $7$, $8$ and $9$. Among these, only the one at numbers $5$ and $7$ are north steps, thus $i(B) = 2$.
    \begin{center}
    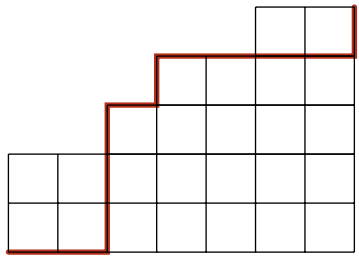
\begin{figure}[ht]
        \begin{tikzpicture}[scale=0.65, line width=.5pt]
        \draw[line width=2pt,BrickRed,line cap=round] (0,0)--(2,0)--(2,1)--(2,3)--(3,3)--(3,4)--(5,4)--(7,4)--(7,5);
        
        \draw (0,0) grid (7,2);
        \draw (2,2) grid (7,3);
        \draw (3,3) grid (7,4);
        \draw (5,4) grid (7,5);
        \end{tikzpicture}
    \caption{$B=\{3,4,5,7,12\}$.}\label{fig:activities}
    \end{figure}
\end{center}
\end{example}

\subsection{The covaluative group}

Some important and useful sources regarding matroid valuations and covaluations are \cite{derksen-fink,ardila-sanchez,eur-huh-larson,ferroni-schroter2}. As we have mentioned in the introduction, the $g$-polynomial is a covaluation under matroid polytope subdivisions. In other words, it behaves additively on the internal faces for matroidal subdivisions. 

More precisely, if $\M$ is a matroid on the ground set $E$ and $\mathscr{P}(\M)$ denotes the base polytope, we define the indicator function of the interior,
    \[\amathbb{1}_{\mathscr{P}(\M)^{\circ}}(x) = \begin{cases}1 & x\in \mathscr{P}(\M)^{\circ}\\
    0 & \text{otherwise}\end{cases}\]
If $\mathcal{S}$ is a subdivision of $\mathscr{P}(\M)$ into matroid polytopes, we have:
    \[ \amathbb{1}_{\mathscr{P}(\M)^{\circ}} = \sum_{\mathscr{P}\in \mathcal{S}^{\circ}}\amathbb{1}_{\mathscr{P}^{\circ}} , \]
where $\mathcal{S}^{\circ}$ stands for the interior faces of the subdivision $\mathcal{S}$. This gives a linear relation between indicator functions of interiors of bases polytopes of matroids; more generally, consider $\mathscr{P}_1,\ldots,\mathscr{P}_m$ base polytopes of matroids on $E$ and satisfying that
    \[ \sum_{j=1}^{m}a_i\amathbb{1}_{\mathscr{P}_i^{\circ}} = 0\]
is the zero function $\mathbb{R}^E\to \mathbb{R}$. Any relation of this type will be called a \emph{covaluative relation}. 

\begin{theorem}[\cite{derksen-fink}]\label{thm:derksen-fink}
    The set $\left\{\amathbb{1}_{\mathscr{P}(\M)^{\circ}}: \M \text{ Schubert matroid on $E$} \right\}$ is a basis of the $\mathbb{Z}$-span of all indicator functions of interiors of matroid polytopes on $E$.
\end{theorem}

If one considers the free $\mathbb{Z}$-module spanned by all matroids on $E$, after modding out by all possible covaluative relations, the resulting abelian group is what we will refer to as the \emph{covaluative group} of matroids on $E$. This group is canonically isomorphic to the $\mathbb{Z}$-span of all indicator functions of interiors of matroid polytopes on $E$, via the identification $\amathbb{1}_{\mathscr{P}(\M)^{\circ}} = [\M]$ where $[\M]$ is the class of $\M$ in the covaluative group. In particular, the covaluative group of matroids admits a basis given by the Schubert matroids.

The above result is possible to derive in an alternative way by building on the work of Eur, Huh and Larson in \cite{eur-huh-larson}: they established an isomorphism between the cohomology of the stellahedral variety and the \emph{valuative group} of matroids (which is defined in a very similar way). A consequence of the above result is that for every matroid $\M$ on $E$ there exists a list of Schubert matroids on $E$, say $\M_1,\ldots,\M_s$ and integers $a_1,\ldots,a_s$, with the property that
    \begin{equation} \label{eq:linear-combination}
    \amathbb{1}_{\mathscr{P}(\M)^{\circ}} = \sum_{j=1}^s a_j \amathbb{1}_{\mathscr{P}(\M_j)^{\circ}}. \end{equation}
Later we will discuss how to recover the coefficients $a_1,\ldots,a_s$ and the Schubert matroids $\M_1,\ldots,\M_s$ by looking at the lattice of cyclic flats of the matroid $\M$, following an idea of Hampe \cite{hampe}.

What is meant by affirming that the $g$-polynomial is a covaluation, is that it acts as a linear map on the covaluative group. In other words, equation \eqref{eq:linear-combination} translates into
    \begin{equation} \label{eq:g-linear-combination}
        g_{\M}(t) = \sum_{j=1}^s a_j\, g_{\M_j}(t).
    \end{equation}

Therefore, we see that the computation of the $g$-polynomial of an \emph{arbitrary} matroid $\M$ is reduced to two problems:
    \begin{itemize}
        \item Computing the $g$-polynomial of Schubert matroids.
        \item Determine the Schubert matroids $\M_1,\ldots,\M_s$ and the coefficients $a_1,\ldots, a_s$ on equation \eqref{eq:linear-combination}.
    \end{itemize}

The first of these two problems is addressed in Sections~\ref{sec:three} and \ref{sec:four}, whereas the last is postponed until Section~\ref{sec:five}.

\section{Delannoy paths and Speyer's invariant}\label{sec:three}

\subsection{Delannoy paths and Schubert matroids}

Whenever $|E|=n$, and we have a Schubert matroid $\M$ on $E=\{i_1,\ldots, i_n\}$ using a linear ordering $i_1<i_2<\cdots<i_n$, we will use the identification of $E$ and its linear ordering with $[n]=\{1,\ldots,n\}$ via $i_j \mapsto j$. There is no harm in using this identification, as we will be trying to calculate the $g$-polynomial, which is invariant under isomorphisms.

In particular, we have access now to the skew shape arising from the lattice path presentation of $\M$. We will introduce a new character into the game. We will be interested in the enumeration of paths within the diagram that not only have north steps $+(0,1)$ and east steps $+(1,0)$, but also \emph{diagonal} steps $+(1,1)$. This type of lattice paths are known in the literature under the name of \emph{Delannoy paths}, or \emph{bilateral Schr\"oder paths}.

A subtlety that arises is that we will be primarily focusing in the case in which our Schubert matroid is loopless and coloopless. This guarantees that the point $(1,1)$ lies inside or on the border of the lattice path representation of $\M$. One can modify the approach slightly to lift this requirement on the loops and coloops, but that implies dealing with an annoying number of cases in some of the proofs below, and the gain is very little as the $g$-polynomial for matroids with loops or coloops vanishes by definition.

\begin{definition}
    Let $\M$ be a loopless and coloopless Schubert matroid on $E$, having upper path $U$. Assume that the rank of $\M$ is $r$ and the size of the groundset is $n$. The \emph{admissible Delannoy paths} associated to $\M$ are all the paths starting at $(1,1)$ and ending at $(n-r,r)$ and having steps of the form $+(1,0)$, $+(0,1)$ and $+(1,1)$, and satisfying the following requirements:
    \begin{enumerate}[(i)]
        \item The paths stay within the lattice path representation of $\M$, i.e., they do not go above $U$.
        \item An intermediate step of the form $+(0,1)$ is valid only if it does not yield a vertical overlap with the upper path $U$.
        \item An intermediate step of the form $+(1,1)$ is valid only when the step $+(0,1)$ is valid.
    \end{enumerate}
\end{definition}

Let us work out some examples and non-examples to get familiar and grasp the meaning of the above definition.

\begin{example}
    Consider $E=\{1,2,\ldots,10\}$ and $U=\{1,2,4,7,9\}$. From left to right, we have in Figure~\ref{fig:delannoy-non-examples} two Delannoy lattice paths that are \emph{not} admissible, followed by two admissible examples. The first path has a vertical overlap with $U$ at the fourth step, and thus violates condition (ii) of the definition. The second path does not overlap $U$ but at the third step performs a $+(1,1)$ movement precisely when it is prohibited, because doing a $+(0,1)$ would yield a vertical overlap, this violates condition (iii). The two remaining examples are admissible: notice that in the third of the four paths we do have an overlap with $U$ but it is not vertical. 

    \begin{center}
    \begin{figure}[ht]
        \begin{tikzpicture}[scale=0.55, line width=.5pt]
        \draw[line width=2pt,BrickRed,line cap=round] (1,1)--(1,2)--(2,2)--(3,3)--(3,3)--(3,4)--(4,4)--(5,4)--(5,5);
        
        \draw (0,0) grid (5,2);
        \draw (1,2) grid (5,3);
        \draw (3,3) grid (5,4);
        \draw (4,4) grid (5,5);
        \end{tikzpicture}\qquad
        \begin{tikzpicture}[scale=0.55, line width=.5pt]
        \draw[line width=2pt,BrickRed,line cap=round] (1,1)--(4,4)--(5,4)--(5,5);
        
        \draw (0,0) grid (5,2);
        \draw (1,2) grid (5,3);
        \draw (3,3) grid (5,4);
        \draw (4,4) grid (5,5);
        \end{tikzpicture}
        \qquad
        \begin{tikzpicture}[scale=0.55, line width=.5pt]
        \draw[line width=2pt,BrickRed,line cap=round] (1,1)-- (2,1) -- (2,3)--(5,3)--(5,5);
        
        \draw (0,0) grid (5,2);
        \draw (1,2) grid (5,3);
        \draw (3,3) grid (5,4);
        \draw (4,4) grid (5,5);
        \end{tikzpicture} 
        \qquad
        \begin{tikzpicture}[scale=0.55, line width=.5pt]
        \draw[line width=2pt,BrickRed,line cap=round] (1,1)-- (3,3) -- (4,3)--(5,4)--(5,5);
        
        \draw (0,0) grid (5,2);
        \draw (1,2) grid (5,3);
        \draw (3,3) grid (5,4);
        \draw (4,4) grid (5,5);
        \end{tikzpicture}
        
    \caption{Two non-examples (left) and two examples (right).}\label{fig:delannoy-non-examples}
    \end{figure}
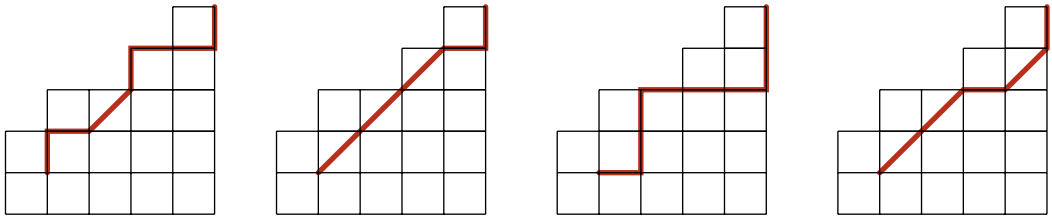
\end{center}
\end{example}

\subsection{Subdivisions into direct sums of series-parallel matroids}

As was pointed out in \cite[Section~5.2]{ferroni-schroter2}, lattice path matroids (and, in particular, Schubert matroids) can be subdivided into direct sums of series-parallel matroids. This can be achieved by performing a sequence of hyperplane splits. Not all matroids enjoy the property of having a subdivision into series-parallel matroids; in fact, some matroids cannot be subdivided at all, for example the graphic matroid $\mathsf{K}_4$ induced by the complete graph on $4$ vertices --- it has $\beta$-invariant equal to $2$, hence it is not series-parallel, and does not admit any subdivisions into smaller matroid polytopes.

\begin{theorem}\label{thm:g-polynomial-delannoy}
    Let $\M$ be a loopless and coloopless Schubert matroid on $E$. There exists a subdivision $\mathcal{S}$ of $\mathscr{P}(\M)$ into base polytopes of series-parallel matroids. Moreover, there is a bijection $\varphi$ between the internal faces of $\mathcal{S}$ with the set of all admissible Delannoy paths associated to $\M$. This bijection has the following properties:
    \begin{enumerate}[\normalfont(i)]
        \item The facets of the subdivision $\mathcal{S}$ correspond to paths with no diagonal steps.
        \item More generally, for each matroid $\N$ corresponding to an internal face of $\mathcal{S}$, the number of connected components $c(\N)$ of $\N$ satisfies that $c(\N)-1$ is the number of diagonal steps of the Delannoy path associated to $\N$.
    \end{enumerate}
\end{theorem}

\begin{proof}
    We will use a result of Chatelain and Ram\'irez-Alfons\'in \cite[Section~2.2]{ChatelainAlfonsin}, also appearing in Bidkhori's thesis \cite[Lemma~4.3.5]{Bidkhori} and \cite[Proposition~5.11]{ferroni-schroter2}. A Schubert matroid admits a subdivision in which all the facets are base polytopes of ``snakes'' (also known as ``border strips''). These correspond to all connected lattice path matroids of rank $r=\rk(\M)$ and size $n=|E|$ whose representation sits inside the representation of $\M$ and do not contain interior lattice points (see, e.g., Figure~\ref{fig:snake-decomposition}). We map each of these snakes bijectively to an admissible Delannoy path without diagonal steps by looking at what the upper path of the snake looks like and removing the first step (which is always $+(0,1)$) and the last step (which is always $+(1,0)$), and placing it with origin at $(1,1)$); cf. Figure~\ref{fig:paths-snakes} to see the admissible Delannoy paths associated to the three snakes of Figure~\ref{fig:snake-decomposition}.
    \begin{center}
    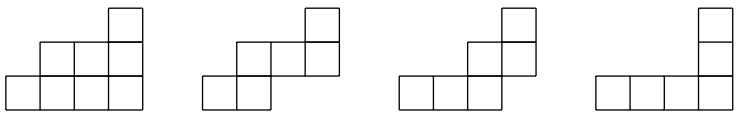
\begin{figure}[ht]
        \begin{tikzpicture}[scale=0.45, line width=.5pt]
        
        \draw (0,0) grid (4,1);
        \draw (1,1) grid (4,2);
        \draw (3,2) grid (4,3);
        
        \end{tikzpicture}\qquad
        \begin{tikzpicture}[scale=0.45, line width=.5pt]
        
        \draw (0,0) grid (2,1);
        \draw (1,1) grid (4,2);
        \draw (3,2) grid (4,3);
        \end{tikzpicture}\qquad
        \begin{tikzpicture}[scale=0.45, line width=.5pt]
        
        \draw (0,0) grid (3,1);
        \draw (2,1) grid (4,2);
        \draw (3,2) grid (4,3);
        
        \end{tikzpicture}\qquad
        \begin{tikzpicture}[scale=0.45, line width=.5pt]
        
        \draw (0,0) grid (4,1);
        \draw (3,1) grid (4,3);
        
        \end{tikzpicture}
    \caption{A Schubert matroid and the three snakes of the described subdivision.}\label{fig:snake-decomposition}
    \end{figure}
    \end{center}
    The admissible Delannoy paths (without diagonal steps) of these three snakes are depicted in Figure \ref{fig:paths-snakes}
    \begin{center}
    \begin{figure}[ht]
        \begin{tikzpicture}[scale=0.45, line width=.5pt]
        \draw[line width=2pt,BrickRed,line cap=round] (1,1)-- (2,1) -- (2,2)--(4,2)--(4,3);
        \draw (0,0) grid (4,1);
        \draw (1,1) grid (4,2);
        \draw (3,2) grid (4,3);
        
        \end{tikzpicture}\qquad
        \begin{tikzpicture}[scale=0.45, line width=.5pt]
        \draw[line width=2pt,BrickRed,line cap=round] (1,1)-- (3,1) -- (3,2)--(4,2)--(4,3);
        \draw (0,0) grid (4,1);
        \draw (1,1) grid (4,2);
        \draw (3,2) grid (4,3);
        
        \end{tikzpicture}\qquad
        \begin{tikzpicture}[scale=0.45, line width=.5pt]
        \draw[line width=2pt,BrickRed,line cap=round] (1,1)-- (4,1) --(4,3);
        \draw (0,0) grid (4,1);
        \draw (1,1) grid (4,2);
        \draw (3,2) grid (4,3);
        
        \end{tikzpicture}
    \caption{Admissible Delannoy paths without diagonals.}\label{fig:paths-snakes}
    \end{figure}
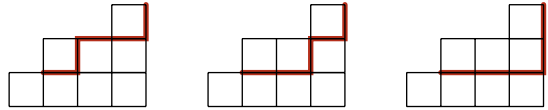
    \end{center}
    Clearly this is a bijection from the facets of the subdivision to the admissible Delannoy paths without diagonals. Now, since the subdivision is a sequence of hyperplane splits, the remaining faces of the subdivision have diagrams which consist of intersections of some pair of the facets\footnote{A caveat is that not all pairs of snakes will yield a diagrammatic intersection corresponding to a valid lattice path matroid of the same rank and cardinality of the original matroid. For example, in a $3\times 3$ rectangle, consider the snake having diagram the first column and the first row and the snake having as diagram the bottom row with the last column. These two intersect in the bottom left and the top right square. This does not correspond to a lattice path matroid.}. For example, the diagrammatic intersection of the last two snakes in Figure~\ref{fig:snake-decomposition} is given by the lattice path matroid depicted on the left of Figure~\ref{fig:intersection-of-snakes}
    \begin{center}
    \begin{figure}[ht]
        \begin{tikzpicture}[scale=0.45, line width=.5pt]
        
        \draw (0,0) grid (3,1);
        \draw (3,1) grid (4,3);
        
        \end{tikzpicture}\qquad\qquad
        \begin{tikzpicture}[scale=0.45, line width=.5pt]
        \draw[line width=2pt,BrickRed,line cap=round] (1,1)-- (3,1)-- (4,2)--(4,3);
        \draw (0,0) grid (4,1);
        \draw (1,1) grid (4,2);
        \draw (3,2) grid (4,3);
        
        \end{tikzpicture}
        \caption{On the left, a face of the subdivision obtained by intersecting two snakes. On the right, its associated admissible Delannoy path.}\label{fig:intersection-of-snakes}
    \end{figure}
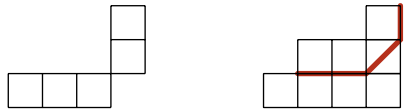
    \end{center}
    For each (valid) intersection of some pair of snakes, the diagonal steps of our Delannoy path will be determined by the steps at which the diagram of the intersection ``breaks'' (see, e.g., the right picture on Figure~\ref{fig:intersection-of-snakes}). That the map described so far between the faces of the subdivision and the admissible Delannoy paths is a bijection, can be deduced from the fact that a Delannoy path is admissible if and only if replacing each $+(1,1)$ by a $+(0,1)$ followed by $+(0,1)$ yields an admissible Delannoy path without diagonal steps, and hence corresponding to a snake, and the same happens with the path obtained by replacing each $+(1,1)$ by a $+(1,0)$ followed by a $+(0,1)$; of course the intersection of the two mentioned snakes would correspond to the admissible Delannoy path we started with.
\end{proof}

\begin{theorem}\label{thm:main-delannoy-g-body}
    Let $\M$ be a loopless and coloopless Schubert matroid. The $g$-polynomial of $\M$ is given by:
        \[ g_{\M}(x) = \sum_{i=1}^{\rk(\M)} c_i\, t^i,\]
    where $c_i$ counts the number of admissible Delannoy paths associated to $\M$ having exactly $i-1$ diagonal steps. 
\end{theorem}

\begin{proof}
    Since all snakes are series-parallel matroids (and they are not single loops/coloops), their $g$-polynomial is $t$. By condition (ii) in the preceding Theorem, each face of the subdivision is in bijection with an admissible Delannoy path and moreover, the number of connected components is enumerated by one plus the number of diagonal steps; this tells that the $g$-polynomial of the corresponding face is $t^c$ where $c$ is one plus the number of diagonal steps. The $g$-polynomial is covaluative under matroid polytope subdivisions, thus the result follows.
\end{proof}

\begin{example}
    Consider the uniform matroid $\U_{r,n}$, it is a Schubert matroid using $U=\{1,\ldots,r\}$ on $E=\{1,\ldots,n\}$. The lattice path presentation of this matroid is a rectangle with vertices $(0,0)$, $(n-r,0)$, $(n-r,r)$ and $(0,r)$. Essentially all Delannoy paths from $(1,1)$ to $(n-r,r)$ (without any restrictions!) are admissible. If we fix the number of diagonal steps, say $i$, we have that each Delannoy path from $(1,1)$ to $(n-r,r)$ corresponds to a words of length $n-2-i$ of symbols $\mathrm{N}$ (north), $\mathrm{E}$ (east) and $\mathrm{D}$ (diagonal), having exactly $i$ occurrences of the symbol $\mathrm{D}$, exactly $r-1-i$ occurrences of the symbol $\mathrm{N}$ and exactly $n-r-1-i$ occurrences of the symbol $\mathrm{E}$. We have $\binom{n-2-i}{i}$ possibilities of where the symbols $\mathrm{D}$ can be positioned, and among the rest we have $\binom{n-2-2i}{r-1-i}$ of where to put the symbols $\mathrm{N}$. In particular,
    \begin{equation}\label{eq:g-uniform}
        g_{\U_{r,n}}(t) = \sum_{i=0}^{\min(r-1,n-r-1)} \binom{n-2-i}{i}\binom{n-2-2i}{r-1-i}\, t^{i+1}.
    \end{equation}
    After expanding the binomial coefficients and some simplifications, this coincides with the formula derived by Speyer in \cite[Proposition~10.1]{speyer}.
\end{example}

A key observation of Speyer in his proof of the case $i=1$ of Conjecture~\ref{conj:f-vector-conj} is that the linear term of the $g$-polynomial is precisely the $\beta$-invariant (see Ardila's ICM survey \cite[p.~13]{ardila-icm22} for a brief outline, or \cite[Theorem~3.1]{speyer-conjecture}). We can retrieve this fact from Theorem~\ref{thm:g-polynomial-delannoy}.

\begin{corollary}
    For every matroid $\M$ on at least $2$ elements we have that the linear term of $g_{\M}(t)$ coincides with the $\beta$-invariant of $\M$.
\end{corollary}

\begin{proof}
    Since the $\beta$-invariant is a covaluation\footnote{It is a valuation too, as it vanishes for disconnected matroids.}, and the linear coefficient of the $g$-polynomial is a covaluation too, by Theorem~\ref{thm:derksen-fink} it suffices to show that the statement that we want to prove holds for Schubert matroids. Using Crapo's formula for the Tutte polynomial,
        \[ T_{\M}(x,y) = \sum_{B\text{ basis}} x^{i(B)}\, y^{e(B)},\]
    (see \cite{crapo-tutte}), since the $\beta$-invariant is given by the coefficient of $x^0y^1$ in the Tutte polynomial, we have 
        \[\beta(\M) = \#\{B\text{ basis of $\M$} : i(B) = 0, e(B)=1\}.\]
    For Schubert matroids, in light of Proposition~\ref{prop:basis-activities-schubert} this means that the $\beta$-invariant is counting lattice paths from $(0,0)$ to $(n-r,r)$ having $+(1,0)$ and $+(0,1)$ steps and starting with a $+(1,0)$ followed by a $+(0,1)$, and not intersecting $U$ in a vertical step. Ignoring the first two steps, this is the same as an admissible Delannoy path without diagonals steps, and hence Theorem~\ref{thm:g-polynomial-delannoy} tells us that this coincides with the linear term of the $g$-polynomial. 
\end{proof}

\begin{remark}
    Although the $g$-polynomial is \emph{not} a specification of the Tutte polynomial, we point out that the Tutte polynomial ``remembers'' more than just the linear term of $g$. A result of Merino, De~Mier and Noy \cite{merino-demier-noy} proves that the Tutte polynomial of a connected matroid does not factor over the integers (nor the complex numbers), hence it encodes the number of connected components of $\M$. This recovers (i) the degree of the least non-zero coefficient of the $g$-polynomial, because $g$ is multiplicative and the linear term of $g$ on a connected matroid is strictly positive (because it coincides with the $\beta$-invariant, which only vanishes for disconnected matroids), and (ii) the evaluation of $g_{\M}(t)$ at $t=-1$, which is $(-1)^{c(\M)}$. It would be very interesting to understand if there are other properties of the $g$-polynomial encoded (in some form) in the Tutte polynomial. 
\end{remark}

\begin{example}
    Let us consider the Catalan matroid of rank $r$ \cite{ardila-catalan}, denoted $\mathsf{Cat}_r$. It has $n=2r$ elements, it is a Schubert matroid, and admits a lattice path presentation as in Figure~\ref{fig:catalan}. In this case, an admissible Delannoy path corresponds to the notion of \emph{Schr\"oder path} under the requirement of not having $+(1,1)$ steps on the diagonal $y=x$. These numbers are listed in the OEIS \hyperlink{https://oeis.org/A033282}{A033282} \cite{oeis} , and coincide with the $f$-vectors of associahedra (see Table~\ref{tab:hvec}). Interestingly, denoting $\widetilde{g}_{\M}(t) := \frac{1}{t} g_{\M}(t)$, for each $r\geq 2$ we have:
        \[ \widetilde{g}_{\mathsf{Cat}_{r}}(t-1) = N_{r-2}(t) = (r-2)\text{-th Narayana polynomial}. \]
    
    \begin{center}
    \begin{figure}[ht]
        \begin{tikzpicture}[scale=0.45, line width=.5pt]
        
        \draw (0,0) grid (1,1);
        \end{tikzpicture}\qquad
        \begin{tikzpicture}[scale=0.45, line width=.5pt]
        
        \draw (0,0) grid (2,1);
        \draw (1,1) grid (2,2);
        
        \end{tikzpicture}\qquad
        \begin{tikzpicture}[scale=0.45, line width=.5pt]
        
        \draw (0,0) grid (3,1);
        \draw (1,1) grid (3,2);
        \draw (2,2) grid (3,3);
        
        \end{tikzpicture}\qquad
        \begin{tikzpicture}[scale=0.45, line width=.5pt]
        
        \draw (0,0) grid (4,1);
        \draw (1,1) grid (4,2);
        \draw (2,2) grid (4,3);
        \draw (3,3) grid (4,4);
        \end{tikzpicture}\qquad
        \begin{tikzpicture}[scale=0.45, line width=.5pt]
        
        \draw (0,0) grid (5,1);
        \draw (1,1) grid (5,2);
        \draw (2,2) grid (5,3);
        \draw (3,3) grid (5,4);
        \draw (4,4) grid (5,5);
        \end{tikzpicture}
    \caption{The lattice path representations of $\mathsf{Cat}_r$ for $r=1,2,3,4,5$}\label{fig:catalan}
    \end{figure}
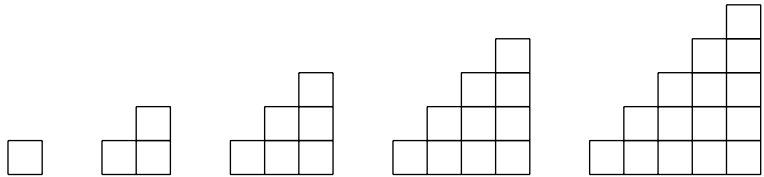
\end{center}
\end{example}

\section{A reinterpretation of the Delannoy paths}\label{sec:four}

Our goal now is to get rid of the lattice path terminology, in order to provide an expression for the $g$-polynomial only in terms of basis activities. That this can be done is a priori plausible due to Proposition~\ref{prop:basis-activities-schubert}.

Let us fix a Schubert matroid $\M$ of size $n$ and rank $r$. Now consider an arbitrary (i.e., not necessarily admissible) Delannoy path $P$ from $(1,1)$ to $(n-r,r)$ that stays inside the lattice path representation of $\M$ . We now describe a criterion to test the admissibility of $P$. Create a basis of $\M$ associated to $P$, denoted by $B_P$, and described as a lattice path as follows: $\path(B_P)$ starts with an east step followed by a north step, then the remaining steps are the same as $P$ but replacing each diagonal step by a north step followed by an east step.

\begin{center}
    \begin{figure}[ht]
        \begin{tikzpicture}[scale=0.55, line width=.5pt]
        \draw[line width=2pt,BrickRed,line cap=round] (1,1)-- (3,3) -- (4,3)--(5,4)--(5,5);
        
        \draw (0,0) grid (5,2);
        \draw (1,2) grid (5,3);
        \draw (3,3) grid (5,4);
        \draw (4,4) grid (5,5);
        \end{tikzpicture}\qquad \qquad
        \begin{tikzpicture}[scale=0.55, line width=.5pt]
        \draw[line width=2pt,blue,line cap=round] (0,0)--(1,0)--(1,1)--(1,2)--(2,2)--(2,3)--(3,3) -- (4,3)--(4,4)--(5,4)--(5,5);
        
        \draw (0,0) grid (5,2);
        \draw (1,2) grid (5,3);
        \draw (3,3) grid (5,4);
        \draw (4,4) grid (5,5);
        \end{tikzpicture}
    \caption{A Delannoy path $P$ in red and the associated basis $B_P$ in blue.}\label{fig:delannoy-to-bases}
    \end{figure}
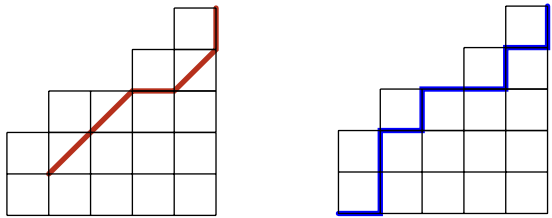
\end{center}

\begin{lemma}
    A Delannoy path $P$ as described above is determined by the basis $B_P$ and the information of where the diagonal steps were performed. Moreover, the path $P$ is admissible if and only if the basis $B_P$ has internal activity $0$ and external activity $1$.
\end{lemma}

\begin{proof}
    This follows readily from the definitions and Proposition~\ref{prop:basis-activities-schubert}.
\end{proof}

Now, the let us fix a basis $B$ having internal activity $0$ and external activity $1$. The idea is to think of $B$ as a putative $B_P$ for some admissible $P$, and list all the possible admissible $P$. To this end, observe that the diagonal steps of such a $P$ are allowed appear only whenever $B$ has a north step followed by an east step (these two steps would be replaced to create a diagonal step in $P$), but only after the first two steps of $B$ which are by definition anchored as always being $+(1,0)$ and then $+(0,1)$. More precisely, this means considering elements $i\in B$ such that $i+1\notin B$ and $i>2$. More succinctly, the possible diagonal steps might appear at the elements of the following set:
    \[ \left\{i\in B : B':=(B\smallsetminus\{i\})\cup\{i+1\}\in \mathscr{B} \text{ and } e(B') = 1, \,i(B')=0\right\}.\]

This motivates the following definition.

\begin{definition}
    Let $\M=(E,\mathscr{B})$ be a matroid on an (ordered) ground set $E$. For simplicity, assume that $E=[n]$ and that the order is given by $1<2<\cdots<n$. For each basis $B$ of $\M$ we define:
    \[ \ent(B):= \#\left\{i\in B : B':=(B\smallsetminus\{i\})\cup\{i+1\}\in \mathscr{B} \text{ and } e(B') = e(B), \,i(B')=i(B)\right\}.  \]
\end{definition}

For example, for the Schubert matroid of Figure~\ref{fig:delannoy-to-bases} and the basis corresponding the path depicted in blue, i.e., $B=\{2,3,5,8,10\}$, the elements contributing to $\alpha(B)$ are $\{3,5,8\}$ so that $\ent(B)=3$.

By leveraging all the preceding observations and notions, we are in position to prove the following statement.

\begin{theorem}\label{thm:main-entropies-body}
    Let $\M$ be a loopless and coloopless Schubert matroid. Then, the $g$-polynomial of $\M$ is given by
    \[ g_{\M}(t) = t\; \sum_{\substack{B\in\mathscr{B}\\e(B)=1\\i(B)=0}} (t+1)^{\ent(B)}. \]
\end{theorem}

\begin{proof}
    Let us fix a basis $B$ having external activity $1$ and internal activity $0$. Each pair of consecutive steps of the form $+(0,1)$ and $+(1,0)$ after the second step can be turned into a diagonal movement, thus producing an admissible Delannoy path. Conversely, each admissible Delannoy path will arise in this way by taking $B=B_P$ and interchanging the necessary ``north + east'' steps to diagonals to recover $P$. Each basis $B$ has exactly $\ent(B)$ occurrences of a north step followed by an east step. Fixing the number of diagonals we want to produce, say $i$, we have $\binom{\ent(B)}{i}$ different admissible Delannoy paths inducing the same $B=B_P$. In particular, Theorem~\ref{thm:g-polynomial-delannoy} translates into:
        \[ g_{\M}(t) = \sum_{B} \sum_{i=0}^{\ent(B)}\binom{\ent(B)}{i} t^{i+1} = \sum_{B} t (t+1)^{\ent(B)},\]
    as desired.
\end{proof}

There is no reason to hope that the same statement will hold for general matroids. In fact, it does not as the reader might verify by considering basically any non-Schubert matroid. What is much more surprising, however, is that not only we see explicitly the non-negativity of the coefficients of $g$-polynomial (predicted by Speyer's result, as Schubert matroids are representable over the complex numbers), but also a stronger phenomenon is true.

\begin{corollary}
    For each matroid, let us denote $\widetilde{g}_{\M}(t) := \frac{1}{t}g_{\M}(t)$. If $\M$ is a Schubert matroid we have that $\widetilde{g}_{\M}(t-1)$ has non-negative coefficients.
\end{corollary}

One informal way of thinking about the above statement is that $\widetilde{g}_{\M}(t)$ behaves as the $f$-vector of a ``Cohen--Macaulay complex''. Due to a result of Speyer \cite[Proposition~6.4]{speyer}, for a connected matroid we always have that $\widetilde{g}_{\M}(-1)=1$. We also have $\widetilde{g}_{\M}(0)=\beta(\M)$. In particular, for an arbitrary connected matroid, the polynomial $\widetilde{g}_{\M}(t-1)$ has constant term equal to $1$ and the sum of its coefficients is the $\beta$-invariant. However, we mention explicitly that the non-negativity property stated in the preceding corollary fails for non-Schubert matroids. The smallest connected example is the Fano matroid. We have $g_{\M}(t) = 3t^3+5t^2+3t$, and $\widetilde{g}_{\M}(t-1)=3t^2-t+1$. A connected representable example is the matroid $\mathsf{R}_8$ (see Table~\ref{tab:hvec}).

\begin{remark}
    This provides an example of a covaluation arising ``in nature'' that is non-negative at all Schubert matroids but fails to be positive in general. If, on the other hand, the reader is interested on valuations having this property, we suggest considering the map $\M\mapsto \rk(\M)+1-\#\{\text{cyclic flats of $\M$}\}$; for Schubert matroids this is trivially non-negative, whereas for matroids in general it will attain negative values (the fact that this is a valuation is not obvious, but follows from the techniques of Derksen and Fink \cite{derksen-fink} or, alternatively, Ferroni and Schr\"oter \cite{ferroni-schroter2}).
\end{remark}

We conjecture, however, that this positivity phenomenon persists for connected matroids that can be subdivided into direct sums of series-parallel matroids. We also raise the challenge of giving a combinatorial characterization of all matroids admitting such subdivisions. 

\begin{conjecture}
    Let $\M$ be a connected matroid whose base polytope admits a subdivision into direct sums of series-parallel matroids. Then $\widetilde{g}_{\M}(t-1)$ has non-negative coefficients.
\end{conjecture}

The general scheme of our proof is possible to extend to arbitrary lattice path matroids, at the expense of imposing additional restrictions to the notion of admissibility on the Delannoy paths. What seems much more challenging is proving that positroids always admit subdivisions into series-parallel matroids and satisfy the non-negativity property of $\widetilde{g}_{\M}(t-1)$. It is not clear to the authors whether the methods of Speyer and Williams in \cite{speyer-williams} might yield a proof of (at least) the first assertion --- i.e, that positroids always can be subdivided into direct sums of series-parallel matroids. We suspect that much interesting combinatorics might be discovered by trying to understand the analogue of our ``admissible Delannoy paths'' via the different combinatorial ways of representing a positroid (e.g., plabic graphs, Grassmann necklaces, decorated permutations, etc. \cite{oh}). The case of transversal and cotransversal matroids is also of interest; notice that the classes of transversal matroids and positroids differ (cf. \cite[Section~6]{marcott}). In \cite{fink-rincon} Fink and Rinc\'on studied matroidal subdivisions of transversal matroid polytopes. One can (perhaps more ambitiously) ask for a characterization of \emph{all} matroids whose base polytope admits a subdivision into series-parallel matroids. 

\section{An algorithm for general matroids}\label{sec:five}

Now we focus on how to compute the coefficients $a_1,\ldots,a_s$ and the Schubert matroids $\M_1,\ldots,\M_s$ in equation~\eqref{eq:linear-combination}. This can be done by following either an approach due to Derksen and Fink \cite{derksen-fink} or, alternatively, Hampe \cite{hampe}. We will follow the approach of Hampe, as it is more convenient to our purposes. 

If $\M$ is a matroid on the ground set $E$ and $\mathscr{P}(\M)$ denotes the base polytope, we define the indicator function,
    \[\amathbb{1}_{\mathscr{P}(\M)}(x) = \begin{cases}1 & x\in \mathscr{P}(\M)\\
    0 & \text{otherwise}\end{cases}\]
If $\mathcal{S}$ is a subdivision of $\mathscr{P}(\M)$ into matroid polytopes, via inclusion-exclusion, we obtain the valuative relation:
    \[ \amathbb{1}_{\mathscr{P}(\M)} = \sum_{\mathscr{P}\in \mathcal{S}^{\circ}} (-1)^{\dim\mathscr{P}(\M)-\dim\mathscr{P}}\amathbb{1}_{\mathscr{P}}, \]
where $\mathcal{S}^{\circ}$ stands for the interior faces of the subdivision $\mathcal{S}$. Notice how in Section~\ref{sec:two} covaluations behaved additively. 

We will denote by $\operatorname{Val}(E)$ the $\mathbb{Z}$-module spanned by all matroids on $E$ and modding out over all valuative relations. The abelian group $\operatorname{Val}(E)$ is canonically isomorphic to the integer span of all indicator functions of matroid polytopes on $E$ via the map $[\M]\mapsto \amathbb{1}_{\mathscr{P}(\M)}$.

If one further identifies classes of matroids in $\operatorname{Val}(E)$ whenever their base polytopes differ by translations, one obtains the polytope algebra of McMullen \cite{mcmullen} for matroids on $E$, which we will denote by $\mathscr{A}_E$ --- here we are leveraging implicitly a result of Derksen and Fink \cite[Theorem~3.5]{derksen-fink}, which proves that strong valuations and weak valuations agree for the class of base polytopes of matroids. We refer to \cite[Appendix~A]{eur-huh-larson} for more details about the interaction of all these concepts in the framework of matroids. 

The take out here is that for a fixed abelian group $A$, any $\mathbb{Z}$-module homomorphism $\operatorname{Val}(E)\to A$ can be thought of as a valuation for matroids on $E$; correspondingly a $\mathbb{Z}$-module homomorphism $\mathscr{A}_E\to A$ stands for a \emph{translation invariant} valuation. Observe that two matroid polytopes on $E$, say $\mathscr{P}(\M)$ and $\mathscr{P}(\N)$, are related by a translation if and only if the matroid $\N$ is obtained from the matroid $\M$ by replacing loops by coloops or viceversa. If a matroid $\M$ has no loops nor coloops, the translation class of the base polytope $\mathscr{P}(\M)$ contains only one polytope.
A result by Derksen and Fink \cite{derksen-fink} guarantees that $\operatorname{Val}(E)$ has as basis the set of all classes of Schubert matroids on $E$, whereas $\mathscr{A}_E$ has as basis the set of all classes of loopless Schubert matroids on $E$.

\begin{definition}
    Let $\M$ be a matroid and let $\mathscr{Z}(\M)$ be its lattice of cyclic flats. The \emph{cyclic chain lattice} of~$\M$ is defined as the lattice $\mathcal{C}_{\mathscr{Z}}(\M)$ whose elements are all the chains of $\mathscr{Z}(\M)$ that contain the minimal $0_{\mathscr{Z}}$  and maximal $1_{\mathscr{Z}}$ element of the lattice $\mathscr{Z}(\M)$; and an additional top element, denoted by $\widehat{\mathbf{1}}$.
\end{definition}

It can be proved that $\mathcal{C}_{\mathscr{Z}}$ is in fact a lattice. Consider the M\"obius function of this poset, as in Stanley \cite[Chapter~3]{stanley-ec1}. To each element $\mathrm{C}\in \mathcal{C}_{\mathscr{Z}}(\M)$ we can associate the number $\lambda_{\mathrm{C}} = -\mu(\mathrm{C},\widehat{\mathbf{1}})$.
Furthermore, since each of these elements $\mathrm{C}$ is a chain of cyclic flats, there is a unique Schubert matroid $\mathsf{S}_{\mathrm{C}}$ whose lattice of cyclic flats coincides with $\mathrm{C}$. Observe that $\M$ and all of the matroids $\mathsf{S}_{\mathrm{C}}$ share the same sets of loops and coloops.

\begin{theorem}\label{thm:hampe}
    Let $\M$ be a matroid without loops and coloops. Then,
    \[ \amathbb{1}_{\mathscr{P}(\M)} = \sum_{\substack{\mathrm{C}\in\mathcal{C}_{\mathscr{Z}}(\M)\\\mathrm{C} \neq \widehat{\mathbf{1}}}} \lambda_{\mathrm{C}}\, \amathbb{1}_{\mathscr{P}(\mathsf{S}_{\mathrm{C}})}.\]
\end{theorem}

\begin{proof}
    Let us denote by $E=\{1,\ldots,n\}$ the ground set of the matroid $\M$. By a result of Hampe \cite[Theorem~3.12]{hampe}, one has an equality in the ``intersection ring of matroids'' $\mathbb{M}_{n}$,
        \begin{equation}\label{eq:hampe}
        B_{\M} = \sum_{\substack{\mathrm{C}\in\mathcal{C}_{\mathscr{Z}}(\M)\\\mathrm{C} \neq \widehat{\mathbf{1}}}} \lambda_{\mathrm{C}}\, B_{\mathsf{S}_{\mathrm{C}}}.
        \end{equation}
    Here $B_{\M}$ stands for the Bergman class of $\M$, and analogously for $B_{\mathrm{S}_\mathrm{C}}$. It follows from Berget, Eur, Spink and Tseng \cite[Section~7]{best} that the assignment $\M \mapsto B_{\M}$ is a translation invariant valuation (alternatively, one can prove this by more elementary means by relying on the ``catenary data'' studied by Bonin and Kung \cite{bonin-kung}). In particular, this tells that there is a $\mathbb{Z}$-module homomorphism $\mathscr{A}_E\to \mathbb{M}_n$ given by $[\M] \to B_{\M}$. Since both rings $\mathscr{A}_E$ and $\mathbb{M}_n$ have bases given all loopless Schubert matroids on $E$, this homomorphism is actually an isomorphism. In particular, equation \eqref{eq:hampe} gives the following equality in $\mathscr{A}_E$:
    \begin{equation}\label{eq:hampe2}
        [\M] = \sum_{\substack{\mathrm{C}\in\mathcal{C}_{\mathscr{Z}}(\M)\\\mathrm{C} \neq \widehat{\mathbf{1}}}} \lambda_{\mathrm{C}}\, [\mathsf{S}_{\mathrm{C}}].
        \end{equation}
    Since $\M$ is loopless and coloopless, then $\mathscr{P}(\M)$ is the unique element in its translation class. Since all the Schubert matroids $\mathsf{S}_{\mathrm{C}}$ appearing in the above sum are also loopless and coloopless, they are unique in their translation classes as well. Hence, we can lift the above statement to $\operatorname{Val}(E)$ (and hence, via the canonical isomorphism $[\M]\mapsto \amathbb{1}_{\mathscr{P}(\M)}$, to the span of indicator functions of matroids), as desired.
\end{proof}

We can turn the above result into a statement about indicator functions of interiors by applying the Euler map of McMullen~\cite{mcmullen}, thus obtaining:

\begin{corollary}
    Let $\M$ be a matroid without loops and coloops. Then,
    \[ \amathbb{1}_{\mathscr{P}(\M)^{\circ}} = \sum_{\substack{\mathrm{C}\in\mathcal{C}_{\mathscr{Z}}(\M)\\\mathrm{C} \neq \widehat{\mathbf{1}}}} (-1)^{c(\M)-c(\mathrm{S}_\mathrm{C})}\lambda_{\mathrm{C}}\, \amathbb{1}_{\mathscr{P}(\mathsf{S}_{\mathrm{C}})^{\circ}}.\]
\end{corollary}

Since the $g$-polynomial is a covaluation, the above statement translates into the $g$-polynomial of an arbitrary matroid $\M$ be:
    \[ g_{\M}(t) = \sum_{\substack{\mathrm{C}\in\mathcal{C}_{\mathscr{Z}}(\M)\\\mathrm{C} \neq \widehat{\mathbf{1}}}} (-1)^{c(\M)-c(\mathrm{S}_\mathrm{C})} \lambda_{\mathrm{C}}\, g_{\mathsf{S}_{\mathrm{C}}}(t).\]

Notice that since the $g$-polynomial is multiplicative under direct sums of matroids, we can restrict ourselves to the case in which $\M$ is connected, i.e., $c(\M)=1$. Also, notice that for a Schubert matroid $\mathsf{S}_{\mathrm{C}}$, being disconnected means that it has loops or coloops; therefore $g_{\mathsf{S}_{\mathrm{C}}}(t)=0$ whenever $c(\mathsf{S}_{\mathrm{C}})>1$. Combining all these observations we obtain:

\begin{corollary}
    Let $\M$ be a connected matroid, $\M\not\cong\U_{0,1},\U_{1,1}$. Then,
     \[ g_{\M}(t) = \sum_{\substack{\mathrm{C}\in\mathcal{C}_{\mathscr{Z}}(\M)\\\mathrm{C} \neq \widehat{\mathbf{1}}}} \lambda_{\mathrm{C}}\, g_{\mathsf{S}_{\mathrm{C}}}(t).\]
\end{corollary}

We provide the following pseudocode that can be used to compute the $g$-polynomial of an arbitrary matroid.

\hfill
   
    \begin{algorithm}[H]
        \SetAlgoNoLine
        \SetKwInOut{Input}{input}\SetKwInOut{Output}{output}
        \Input{An arbitrary matroid $\M$}
        \Output{The $g$-polynomial of $\M$.}
        
        {\normalfont \textbf{Function}} {\normalfont \texttt{g\_polynomial}}$(\M)$:
        \\
            \Indp{
                \If{$\M$ {\normalfont has loops or coloops}}{
                    \Return 0
                }
                \If{$\M$ {\normalfont is disconnected}}{
                    $\operatorname{ans} = 1$
                    
                    \For{$\N$ {\normalfont connected component of $\M$}}{
                        $\operatorname{ans} = \operatorname{ans}\,\cdot \,\texttt{g\_polynomial}(\N)$
                        
                    }
                    \Return $\operatorname{ans}$
                }
                \Else{
                    $\operatorname{ans} = 0$\;
                    
                    \For{$\mathrm{C}\in \mathcal{C}_{\mathscr{Z}}(\M)$}{
                        $\lambda = -\mu(\mathrm{C},\widehat{\mathbf{1}})$
                        
                        $\mathsf{S}_{\mathrm{C}} = \text{Schubert matroid associated to $\mathrm{C}$}$

                        $\operatorname{aux} = 0$
                        
                        \For{\normalfont{$B$ basis of $\mathsf{S}_{\mathrm{C}}$}}{
                            \If{{\normalfont $e(B)=1$ and $i(B)=0$}}{
                                $\operatorname{aux} = \operatorname{aux} + \; t\cdot (t+1)^{\ent(B)}$
                                
                            }
                        }
                        $\operatorname{ans} = \operatorname{ans} + \lambda\cdot \operatorname{aux}$
                        
                    }
                    \Return $\operatorname{ans}$
                }
            }
    \end{algorithm}
\vspace{1cm}

Together with this manuscript, the reader might find a zip file with an implementation on \texttt{SAGE} of the above algorithm. This can be used to compute within some minutes the $g$-polynomial of all matroids up to $9$ elements. We included in Table~\ref{tab:hvec} the $g$-polynomials of several matroids, calculated using the above procedure, the notation is that of Oxley's catalogue \cite{oxley}, with one extra graphic matroid that we denote by $\K_{1,2,3}$ and comes from the complete tripartite graph with parts of sizes $1$, $2$ and $3$. Some of these calculations had been done by Speyer in \cite{speyer} using the original $\mathrm{K}$-theoretic framework; of course, in all such cases our results coincide with his.

\let\maybemidrule=\relax
\begin{table}[htbp]
\caption{$g$-polynomials of some matroids} 
 \label{tab:hvec}
\begin{centering}
\tiny
\def\arraystretch{1.6}
\begin{tabular}{lccll}
\toprule
& \normalsize size & \normalsize rank & \normalsize $g$-polynomial & \normalsize $\widetilde{g}(t-1)$\\
\midrule

 $\mathsf{K}_4$ & $6$ & $3$ & $t^{3} + 2 t^{2} + 2 t$                            & $ t^2+1$ \\
 \maybemidrule  
 
 $\mathsf{K}_5$ & $10$ & $4$ & $5 t^{4} + 15 t^{3} + 15 t^{2} + 6 t$                            & $ 5t^3+1$ \\
 \maybemidrule    

 Fano & $7$ & $3$ & $ 3 t^{3} + 5 t^{2} + 3 t $ & $ 3 t^{2} - t + 1 $ \\
 \maybemidrule    

 NonFano & $7$ & $3$ & $ 3 t^{3} + 6 t^{2} + 4 t $ & $ 3 t^{2} + 1 $ \\
 \maybemidrule    

 V\'amos & $8$ & $4$ & $ t^{4} + 12 t^{3} + 25 t^{2} + 15 t $ & $ t^{3} + 9 t^{2} + 4 t + 1 $ \\
 \maybemidrule  
 
 $\mathsf{O}_7$ & $7$ & $3$ & $ 2 t^{3} + 5 t^{2} + 4 t $ & $ 2 t^{2} + t + 1 $ \\
 \maybemidrule                                                           
 $AG(3,2)'$ & $8$ & $4$ & $ t^{4} + 12 t^{3} + 17 t^{2} + 7 t $ & $ t^{3} + 9 t^{2} - 4 t + 1 $ \\
 \maybemidrule
 
 $\mathsf{F}_8$ & $8$ & $4$ & $ t^{4} + 12 t^{3} + 18 t^{2} + 8 t $ & $ t^{3} + 9 t^{2} - 3 t + 1 $ \\
 \maybemidrule   

$\mathsf{L}_8$ & $8$ & $4$ & $ t^{4} + 12 t^{3} + 22 t^{2} + 12 t $ & $ t^{3} + 9 t^{2} + t + 1 $ \\
\maybemidrule

$\mathsf{P}_8$ & $8$ & $4$ & $ t^{4} + 12 t^{3} + 20 t^{2} + 10 t $ & $ t^{3} + 9 t^{2} - t + 1 $ \\
\maybemidrule

 $\mathsf{R}_8$ & $8$ & $4$ & $ t^{4} + 12 t^{3} + 18 t^{2} + 8 t $ & $ t^{3} + 9 t^{2} - 3 t + 1 $ \\
 \maybemidrule   

 $\mathsf{T}_8$ & $8$ & $4$ & $ t^{4} + 12 t^{3} + 19 t^{2} + 9 t $ & $ t^{3} + 9 t^{2} - 2 t + 1 $ \\
\maybemidrule
 $\mathsf{K}_{3,3}$ & $9$ & $5$ & $ 4 t^{4} + 12 t^{3} + 12 t^{2} + 5 t $ & $ 4 t^{3} + 1 $ \\
 \maybemidrule                                                                                                              
 TicTacToe & $9$ & $5$ & $ 4 t^{4} + 30 t^{3} + 52 t^{2} + 27 t $ & $ 4 t^{3} + 18 t^{2} + 4 t + 1 $ \\
 \maybemidrule                                                             
 Block 9--4 & $9$ & $4$ & $ 4 t^{4} + 30 t^{3} + 42 t^{2} + 17 t $ & $ 4 t^{3} + 18 t^{2} - 6 t + 1 $ \\
 \maybemidrule                                                      
 
 Block 10--5 & $10$ & $5$ & $ t^{5} + 20 t^{4} + 90 t^{3} + 104 t^{2} + 34 t $ & $ t^{4} + 16 t^{3} + 36 t^{2} - 20 t + 1 $ \\
 \maybemidrule                                                             
 Pappus & $9$ & $3$ & $ 10 t^{3} + 21 t^{2} + 12 t $ & $ 10 t^{2} + t + 1 $ \\
 
 \maybemidrule                                                       
 NonPappus & $9$ & $3$ & $ 10 t^{3} + 22 t^{2} + 13 t $ & $ 10 t^{2} + 2 t + 1 $ \\
 \maybemidrule                                                             
                                                   
 $AG(3,2)$ & $8$ & $4$ & $ t^{4} + 12 t^{3} + 16 t^{2} + 6 t $ & $ t^{3} + 9 t^{2} - 5 t + 1 $ \\
 \maybemidrule                                                             
 $AG(2,3)$ & $9$ & $3$ & $ 10 t^{3} + 18 t^{2} + 9 t $ & $ 10 t^{2} - 2 t + 1 $ \\
 
 \maybemidrule

 $\mathsf{J}$ & $8$ & $4$ & $ t^{4} + 6 t^{3} + 10 t^{2} + 6 t $ & $ t^{3} + 3 t^{2} + t + 1 $ \\
 
\maybemidrule                                                  
 $\mathsf{W}_3$ & $6$ & $3$ & $ t^{3} + 2 t^{2} + 2 t $ & $ t^{2} + 1 $ \\
 \maybemidrule                                                             
 $\mathsf{W}_4$ & $8$ & $4$ & $ t^{4} + 4 t^{3} + 5 t^{2} + 3 t $ & $ t^{3} + t^{2} + 1 $ \\
 \maybemidrule                                                             
 $\mathsf{W}_5$ & $10$ & $5$ & $ t^{5} + 5 t^{4} + 10 t^{3} + 9 t^{2} + 4 t $ & $ t^{4} + t^{3} + t^{2} + 1 $ \\
 \maybemidrule                                                             
 $\mathsf{K}_{1,2,3}$ & $11$ & $5$ & $ 2 t^{5} + 12 t^{4} + 24 t^{3} + 20 t^{2} + 7 t $ & $ 2 t^{4} + 4 t^{3} + 1 $ \\  
 \maybemidrule                                                             
 $\mathsf{Cat}_1$ & $2$ & $1$ & $ t $ & $ 1 $ \\

 \maybemidrule                                                             
 $\mathsf{Cat}_2$  & $4$ & $2$ & $ t $ & $ 1 $ \\
 \maybemidrule                                                             
 $\mathsf{Cat}_3$  & $6$ & $3$ & $ t^{2} + 2 t $ & $ t + 1 $ \\

 \maybemidrule                                                             
 $\mathsf{Cat}_4$ & $8$ & $4$ & $ t^{3} + 5 t^{2} + 5 t $ & $ t^{2} + 3 t + 1 $ \\
 
 \maybemidrule                                                             
 $\mathsf{Cat}_5$ & $10$ & $5$ & $ t^{4} + 9 t^{3} + 21 t^{2} + 14 t $ & $ t^{3} + 6 t^{2} + 6 t + 1 $ \\
\bottomrule
\end{tabular}
\end{centering}
\end{table}

\section*{Acknowledgments}

I am grateful to Chris Eur for encouraging me to write this paper, and to Christian Krattenthaler for conversations about lattice path enumeration that triggered a lot of interest in me. I also benefited from several useful comments and remarks by Alex Fink, Matt Larson, David Speyer, and two anonymous referees, all of whom I thank. Last but not least, I thank my partner, Camilla, for her patience with me throughout the writing process of this paper, and for gifting me with her love and company. This article is dedicated to our first son, Bruno Ferroni, who was born two days after this paper was submitted.

\bibliographystyle{amsalpha}
\bibliography{bibliography}

\end{document}